\documentclass[11pt]{article}
\usepackage[margin=1in]{geometry}
\usepackage{amsmath,amssymb,amsthm,mathtools}
\usepackage{mathrsfs}
\usepackage[T1]{fontenc}
\usepackage{lmodern}
\usepackage{pict2e}
\usepackage[colorlinks=true,allcolors=blue]{hyperref}
\hypersetup{pdftitle={Integral face-edge duality for knot triangulations},pdfauthor={Sai Kambampati}}
\newtheorem{theorem}{Theorem}[section]
\newtheorem{proposition}[theorem]{Proposition}
\newtheorem{lemma}[theorem]{Lemma}
\newtheorem{corollary}[theorem]{Corollary}
\theoremstyle{definition}
\newtheorem{example}[theorem]{Example}
\theoremstyle{remark}
\DeclareMathOperator{\im}{im}
\DeclareMathOperator{\rank}{rank}
\DeclareMathOperator{\nullity}{nullity}
\DeclareMathOperator{\coker}{coker}
\DeclareMathOperator{\diag}{diag}
\newcommand{\R}{\mathbb R}
\newcommand{\Z}{\mathbb Z}
\newcommand{\Q}{\mathbb Q}
\newcommand{\Ct}{\overline C}
\newcommand{\Rt}{\overline R}
\newcommand{\eps}{\varepsilon}
\title{Integral face--edge duality for knot triangulations}
\author{Sai Kambampati\\\small \href{mailto:sai@kambampati.studio}{sai@kambampati.studio}}
\date{October 2, 2026}
\begin{document}
\maketitle
\begin{abstract}
For an ordered ideal triangulation of a knot exterior in $S^3$, we give a cochain construction relating its face matrix $C$ to its preferred-longitude gluing matrices $A,B$. The obstruction to uniqueness is the first cohomology of the ideal compactification. Its vanishing yields nullity doubling and the FAMED matrix identity, including a constrained formulation for singular face matrices. A separate rank argument under a strict angle structure supplies a proposed proof of Wong's generalized FAMED conjecture. Over the integers, the construction gives a direct-sum decomposition of $\ker_{\Z}B$ and an exact sequence from the two face cokernels to the gluing cokernel, with quotient $\Z/2\Z$. In the unhalved longitude convention this implies $|\det B|=2|\det C|^2$. We also obtain cyclic face torsion, order-reversal invariance of Smith normal form, and affine face coordinates for fixed lifted gluing systems. The determinant-one assertion and the full Andersen--Kashaev volume conjecture remain open here.
\end{abstract}

\section{Introduction}
The gluing equations of an ideal triangulation describe how the shapes of its tetrahedra fit around edges and peripheral curves. In the Teichm\"uller TQFT construction, a second collection of matrices records relations among variables attached to faces. The FAMED conditions compare these two descriptions. Ben Aribi and Wong introduced these conditions in their study of the Andersen--Kashaev volume conjecture \cite{BW}; Ben Aribi, Guilloux and Wong subsequently investigated them computationally \cite{BGW}. Wong's generalized conditions include triangulations whose face matrix is singular \cite{Wong}.

The question considered here is why the face matrices and the gluing matrices should satisfy the same constraints. We associate an additive edge cochain to a pair of face assignments, one for each direction of the vertex ordering. Differences of this cochain give the homogeneous shape variables. Face identifications force the edge equations. A separate cocycle on the dual spine shows that the preferred-longitude equation holds as well. The obstruction to recovering the face assignments uniquely is a cohomology group of the ideal compactification; this group vanishes for a knot exterior in $S^3$.

This construction gives the rational correspondence below, including singular face matrices. Integral lifts then relate kernels and cokernels over $\Z$. Finally, a combinatorial rank argument under a strict angle structure supplies the additional condition in Wong's generalized FAMED conjecture. The analytic consequences are applications of Wong's existing theorem. The determinant-one assertion in the original FAMED conjecture requires a further argument and remains open here.

\subsection{Conventions and notation}
Let $M$ be the compact exterior of a knot in $S^3$, and let $\mathcal T$ be an ordered ideal triangulation of its interior with $N$ tetrahedra. The triangulation is formed from abstract tetrahedra by pairing their face occurrences and removing the vertices. Paired faces may belong to the same tetrahedron, and distinct local edges may represent the same quotient edge. All faces are paired; there is one ideal vertex, $2N$ quotient faces and $N$ quotient edges. ``Ordered'' means that each abstract tetrahedron has a vertex order and every face pairing preserves the induced order on its three corners. These corners remain distinct in the abstract face even though all quotient vertices coincide. Fix an orientation of $M$. The sign of tetrahedron $i$ relative to its vertex order $0<1<2<3$ is $\eps_i\in\{1,-1\}$. If faces opposite vertices $k,l$ in tetrahedra $i,j$ are paired, then
\begin{equation}\eps_i(-1)^k=-\eps_j(-1)^l.\label{eq:facesign}\end{equation}
Let $X_k$ be the $N\times2N$ incidence matrix recording the quotient face opposite local vertex $k$. Set
\begin{equation}
 C=\begin{pmatrix}X_0-X_1+X_2\\X_2-X_3\end{pmatrix},\quad
 E=\diag(\eps_i),\quad D=\begin{pmatrix}0\\E\end{pmatrix},\quad R=X_0,
 \quad \Delta=\frac{E+I}{2}.\label{eq:face}
\end{equation}
In the notation of \cite{BW,Wong}, our matrices $(C,D,R,E)$ are $(\mathscr A,\mathscr B,\mathscr X_0,\mathscr E)$. In \cite{BGW}, they are $(\mathcal A,\mathcal B\mathcal E,\mathcal X_0,\mathcal E)$: that paper defines $\mathcal B=\binom{0}{I}$, whereas here $D=\binom{0}{E}$. Our $A,B$ below agree with the gluing matrices $\mathbf A,\mathbf B$ in these sources. Each row of $X_k$ has one entry equal to $1$, in the column of the corresponding quotient face. Distinct local faces can give the same column. Local orders, rather than a global ordering of ideal vertices, are the data being fixed.

Use the shape placement of those papers. The following table also gives
the dictionary with the quadrilateral types in \cite[Section 3.1]{GHHHR}:
each quadrilateral faces the indicated pair of opposite edges.
\[
\begin{array}{c|ccc}
\text{opposite-edge pair}&01\mid23&02\mid13&03\mid12\\\hline
\eps_i=+1&z&z''&z'\\
\eps_i=-1&z&z'&z''
\end{array}
\]
For a quotient edge $e$ and tetrahedron $i$, let
$(G_{\rm edge})_{ei}$ be the number of local edge occurrences of type
$z$ in tetrahedron $i$ that represent $e$. Define
$G'_{\rm edge},G''_{\rm edge}$ in the same way for $z',z''$.
All three full edge matrices are $N\times N$; repeated occurrences
are counted separately. Put
\[
A_{\rm edge}=G_{\rm edge}-G'_{\rm edge},\qquad
B_{\rm edge}=G''_{\rm edge}-G'_{\rm edge}.
\]

Choose an oriented simple closed curve $\lambda$ on $\partial M$
whose class is a primitive generator of
$\ker[H_1(\partial M;\Z)\to H_1(M;\Z)]$.
This is the preferred longitude, with either orientation allowed.
Represent it normally in the cusp triangulation: each arc crosses a
triangle between distinct sides and cuts off one corner. Orient the
cusp by the boundary orientation of $M$. An arc contributes $+1$ if
the cut-off corner lies to its left and $-1$ if it lies to its right.
Its contribution is assigned to the shape type of the tetrahedron
edge at that corner. Summing all such contributions in tetrahedron
$i$ defines the three integer entries $g_{\lambda i}$,
$g'_{\lambda i},g''_{\lambda i}$. Section~\ref{sec:peripheral}
expresses this turning sign in local vertex coordinates.

Delete the same one edge row from $G_{\rm edge},G'_{\rm edge},
G''_{\rm edge}$ and append the respective longitude rows to obtain
$N\times N$ matrices $G,G',G''$. Put
\begin{equation}A=G-G',\qquad B=G''-G'.\label{eq:nz}\end{equation}
When simplifying a peripheral formula, one may add integral combinations
of edge rows and of the tetrahedral rows $(1,1,1)$, carrying the lifted
logarithmic constants with them as in Section~\ref{sec:affine}.
No rescaling is made. Tetrahedral rows vanish after prime elimination;
edge-row additions are integral row operations on the retained system,
since the eliminated full edge rows sum to zero.
\begin{lemma}[Rank and symplectic identities]\label{lem:nzfacts}
The full eliminated edge block $(A_{\rm edge}\mid B_{\rm edge})$
has rank $N-1$, and its row relations are generated by the sum of all
rows being zero.
After one edge row is deleted and the row for any nonzero class in
$H_1(\partial M;\Q)$ is appended, the resulting matrices satisfy
\begin{equation}
\rank(A\mid B)=N,\qquad AB^{\mathsf T}=BA^{\mathsf T}.
\label{eq:nzfacts}
\end{equation}
These facts require no shape solution or angle structure.
\end{lemma}
\begin{proof}
Apply the combinatorial complex in \cite[Theorem~4.1]{Neumann} to the
ideal compactification. At the unique ideal vertex $v$, Neumann's
vertex-to-edge map is $\alpha(v)=2\sum_e e$, because every edge has
both endpoints at $v$. Rational exactness at the edge group therefore
makes the kernel of the edge map the span of $\sum_e e$. Thus its image has
dimension $N-1$. The theorem makes this image isotropic in the
$2N$-dimensional space of tetrahedral edge-pair coordinates. The
signed-corner map of \cite[Lemma~4.3]{Neumann} injects
$H_1(\partial M;\Q)$ into the symplectic orthogonal of that image,
modulo the image itself. Therefore a nonzero peripheral class adds an
independent row orthogonal to all edge rows. Its pairing with itself is
zero because the pairing is alternating. Eliminating the prime
coordinate in each tetrahedron uses the relation
$e_z+e_{z'}+e_{z''}=0$ between its opposite-edge generators:
\[
g e_z+g'e_{z'}+g''e_{z''}
=(g-g')e_z+(g''-g')e_{z''}.
\]
In the shape placement above this writes the rows as $(A\mid B)$;
their isotropy is precisely the second identity in \eqref{eq:nzfacts}.
The raw edge exponent rows instead sum to $(2,2,2)$ in every
tetrahedron, which becomes zero after this elimination.
\end{proof}
For the original geometric setting, see \cite[Theorem 2.2 and
Propositions 2.3 and 2.5]{NZ}. The combinatorial formulation used here
does not require a geometric shape solution. The same coordinate
dictionary is used integrally in Section~\ref{sec:integral}.
The preferred longitude is nonzero in $H_1(\partial M;\Q)$, as required
by this lemma, but maps to zero in $H_1(M;\Q)$. Its vanishing in the
interior homology is a different fact, used in
Section~\ref{sec:peripheral}.
Our homogeneous shape coordinates are
\begin{equation}(\delta\log z,\delta\log z',\delta\log z'')=(t,s-t,-s),\label{eq:shapes}\end{equation}
so their gluing relation is $At-Bs=0$.
The peripheral row comes from the full, unhalved signed-corner sum,
with the allowed simplifications above, as in \cite{BGW,Wong}.
Primitivity refers to the class $[\lambda]$, not to
the integer row, which need not be primitive. The corresponding
quadrilateral class has boundary $2[\lambda]$ in the normalization
of \cite[Section 7.1]{GHHHR}; the integral argument below explains
this factor. Dividing the row by two would change the integral
normalization and the determinant statement. Reversing the global
turning convention negates the peripheral row and does not change
the relation or the absolute determinant. The variables in
\eqref{eq:shapes} are formal homogeneous logarithmic coordinates.
Their use does not assume that a geometric shape solution has been chosen.

We use $X=M/\partial M$ for the ideal compactification. Unless a coefficient ring is displayed, kernels, images, ranks and nullities are over $\Q$, or over $\R$ after extension of scalars; the ranks are the same. For an integer matrix $T$ with $m$ rows, $\coker T$ means $\Z^m/T\Z^{\operatorname{cols}(T)}$. A sublattice is \emph{primitive} if its quotient in the ambient lattice is torsion-free. The bars in $\Ct$ and $\Rt$ denote reversal of the local orders, not complex conjugation.

\begin{center}
\begin{tabular}{ll}
\hline
Symbol & Meaning and size\\
\hline
$C,\Ct$ & Original and reversed face matrices, $2N\times2N$\\
$D$ & Signed input matrix for face constraints, $2N\times N$\\
$R,\Rt$ & Face selectors $X_0,X_3$, of size $N\times2N$\\
$E,\Delta$ & $\diag(\eps_i)$ and $\diag((1+\eps_i)/2)$\\
$A,B$ & Edge and preferred-longitude gluing matrices, $N\times N$\\
$n$ & $\nullity C$\\
$W$ & Common solvability space in Section~\ref{sec:relation}\\
\hline
\end{tabular}
\end{center}

\subsection{Main statements}

The construction underlying the results is an explicit correspondence
between pairs of compatible face assignments and the homogeneous
gluing relation. Define
\[
\mathcal F_{\Q}=\left\{(x,y)\in(\Q^{2N})^2:
\begin{aligned}
 (X_0-X_1+X_2)x&=0,\\
 (X_3-X_2+X_1)y&=0,\\
 (X_2-X_3)x&=(X_1-X_0)y
\end{aligned}\right\}.
\]
Sections~\ref{sec:pairing}--\ref{sec:injectivity} prove that the map
\[
\begin{gathered}
\Phi_{\Q}:\mathcal F_{\Q}\longrightarrow\ker_{\Q}(A\mid-B),
\qquad (x,y)\longmapsto(t,s),\\
t=E(X_2-X_3)x,\qquad s=X_0x+X_3y+\Delta t.
\end{gathered}
\]
is an isomorphism. In particular, every formal solution $(t,s)$ has
a unique pair of compatible face assignments. The matrix identities
below are consequences of this correspondence; its complexification
is used in Section~\ref{sec:affine}.

\begin{theorem}[Rational face--edge duality]\label{thm:main}
With these conventions, put $n=\nullity C$. Then
\begin{enumerate}
\item $\nullity B=2n$.
\item If $C$ is invertible, then $B$ is invertible and
\begin{equation}B^{-1}A=RC^{-1}D+(RC^{-1}D)^{\mathsf T}+\Delta.\label{eq:famed}\end{equation}
\item Without invertibility, the face constraint space and its symmetric form coincide with the corresponding Neumann--Zagier relation, as specified in Section~\ref{sec:singular}.
\end{enumerate}
No angle structure is needed for these rational assertions.
\end{theorem}

\begin{theorem}[Integral kernel and cokernel duality]\label{thm:integral}
Let $\Ct$ denote the face matrix with all vertex orders reversed, and set $\Rt=X_3$. With integer kernels and cokernels, one has
\begin{equation}\ker_{\Z}B=R\ker_{\Z}C\ \oplus\ \Rt\ker_{\Z}\Ct.\label{eq:integerkernel}\end{equation}
There is a short exact sequence
\begin{equation}0\longrightarrow\coker C\oplus\coker\Ct\xrightarrow{\ \iota\ }\coker B\longrightarrow\Z/2\Z\longrightarrow0,\label{eq:cokersequence}\end{equation}
where $\iota(([Dt]_C,[Dt]_{\Ct}))=[At]_B$. Every element of the domain admits the indicated representation. The two face matrices have identical Smith normal forms. No splitting of \eqref{eq:cokersequence} is asserted.
\end{theorem}

\begin{theorem}[Determinant relation]\label{thm:det}
With the same hypotheses and the unhalved preferred-longitude convention,
\begin{equation}|\det B|=2|\det C|^2.\label{eq:det}
\end{equation}
\end{theorem}

For geometric triangulations, parts (1) and (2) of Theorem~\ref{thm:main} give parts (1) and (3) of Conjecture~4.1 in \cite{BGW}. The determinant magnitudes asserted in part (2) of that conjecture are separate. The strict rank inequality needed for the following consequence is proved in Section~\ref{sec:geometric}.

\begin{theorem}[Generalized FAMED]\label{thm:generalized}
Every ordered ideal triangulation of a hyperbolic knot complement in $S^3$ admitting a strict angle structure is generalized FAMED with respect to its preferred longitude, in the sense of Definition~1.1 of \cite{Wong}.
\end{theorem}
This is Conjecture~1.2 of \cite{Wong}. It does not assert the extra longitude--meridian conditions of Definition~1.5 there.

\begin{theorem}[Cyclic face torsion]\label{thm:cyclic}
The torsion subgroup of $\coker C$ is cyclic. If $C$ is nonsingular, $\coker C\cong\Z/d\Z$ with $d=|\det C|$. Section~\ref{sec:embedded} realizes $d$ as the divisibility of a filling-core class in the first homology of an explicitly constructed embedded complex. These assertions also hold for an ordered ideal triangulation of an integral homology solid torus.
\end{theorem}

\subsection{Relation to earlier work and proof structure}
The map $x\mapsto Rx$ in \eqref{eq:integerkernel} gives a kernel correspondence of the kind requested in Wong's Problem~2.1 of the 2026 RIMS problem collection \cite{RIMS}, in the preferred-longitude knot setting. Section~\ref{sec:affine} gives a face-variable formulation relevant to Problem~2.3 for each fixed lifted logarithmic system. The construction uses gluing matrices to choose an affine origin; it is not a presentation obtained from the face-pairing data alone. No existence of geometric solutions follows from this change of variables.

The established inputs are the Neumann--Zagier symplectic and rank results \cite{NZ,Neumann}, the integral normal-surface lattice theorem of Garoufalidis, Hodgson, Hoffman and Rubinstein \cite{GHHHR}, and Wong's analytic theorem \cite{Wong}. The contribution proposed here is the explicit correspondence between their gluing data and the face matrix $C$, together with its integral and topological consequences. The reversal invariance discussed in Section~3.3 of \cite{BGW} is refined to an equality of Smith normal forms.

The proof is organized as follows. Sections~\ref{sec:pairing}--\ref{sec:relation} construct the face relation and compute its dimension. Sections~\ref{sec:local}--\ref{sec:peripheral} identify it with the full edge-and-longitude relation. Section~\ref{sec:injectivity} proves uniqueness of the face lift, and Section~\ref{sec:singular} extracts the constrained symmetric form when $C$ is singular. Section~\ref{sec:geometric} proves positive gluing rank under strict angles and applies Wong's theorem. The affine formulation is followed by integral reversal, lattice, and embedded-complex arguments in Sections~\ref{sec:reversal}--\ref{sec:embedded}.

\begin{example}[The figure-eight knot and the peripheral choice]\label{ex:figureeight}
Use the two-tetrahedron triangulation and the face labels of Figure~1 in \cite{RIMS}. Number its four quotient faces from $0$ to $3$. The local face lists are $(1,3,2,0)$ and $(2,0,1,3)$, and the tetrahedron signs are $+1,-1$. In our notation,
\[
C=\begin{pmatrix}
0&1&1&-1\\-1&1&1&0\\-1&0&1&0\\0&1&0&-1
\end{pmatrix},\qquad
D=\begin{pmatrix}0&0\\0&0\\1&0\\0&-1\end{pmatrix},\qquad
R=\begin{pmatrix}0&1&0&0\\0&0&1&0\end{pmatrix}.
\]
The retained edge equation and the preferred longitude give
\[
A=\begin{pmatrix}1&-2\\0&4\end{pmatrix},\qquad
B=\begin{pmatrix}-1&-1\\0&2\end{pmatrix}.
\]
Consequently $RC^{-1}D=\diag(-1,1)$ and $\Delta=\diag(1,0)$, so both sides of \eqref{eq:famed} equal $\diag(-1,2)$. The homogeneous relation is $s_1=-t_1$, $s_2=2t_2$. Also $\det C=-1$ and $\det B=-2$, in agreement with \eqref{eq:det}.

Replacing the longitude by the meridian changes $B$ to
\[B_\mu=\begin{pmatrix}-1&-1\\1&1\end{pmatrix}.
\]
This matrix is singular although $C$ is invertible. Thus the preferred-longitude hypothesis cannot be replaced by an arbitrary peripheral curve. The affine constants for the same example are computed in Section~\ref{sec:affine}.
\end{example}

\section{The pairing and reversal}\label{sec:pairing}
The proof uses two face systems, one for the chosen order and one for
its reverse. Their interaction is encoded by a bilinear pairing on
local face values. All vector spaces in Sections~\ref{sec:pairing}--\ref{sec:geometric}
may be taken over $\Q$; we use $\R$ to keep the orthogonality notation
familiar.

Reverse every local order, retaining the quotient-face and tetrahedron
labels. The permutation $0123\mapsto3210$ is even, so it preserves the
tetrahedron sign. The reversed face matrix and restriction map are
\[
\Ct=\begin{pmatrix}X_3-X_2+X_1\\X_1-X_0\end{pmatrix},
\qquad \Rt=X_3.
\]
Let $V=\R^{4N}$ denote the space of local face assignments. If $f(i,k)$
is the quotient face represented by local face $k$ of tetrahedron $i$,
the map from quotient-face coordinates to local coordinates is
\[
\jmath:\R^{2N}\longrightarrow V,\qquad
\jmath(x)_{i,k}=x_{f(i,k)}=(X_kx)_i.
\]
It is injective, and its image is the subspace $F$ of assignments
constant on paired faces. Equip $V$ with
the nondegenerate symmetric pairing
\begin{equation}
\langle x,y\rangle=
\sum_i\eps_i(x_{i0}y_{i0}-x_{i1}y_{i1}+x_{i2}y_{i2}-x_{i3}y_{i3}).
\label{eq:pair}
\end{equation}
The subspace $F$ has dimension $2N$.
Equation~\eqref{eq:facesign} cancels the two contributions
of each quotient face, so $\langle F,F\rangle=0$ and hence $F=F^\perp$.

The local homogeneous constraint planes are
\[
L_i=\{(u,u+v,v,v):u,v\in\R\},\qquad
\overline L_i=\{(p,p,p+w,w):p,w\in\R\}.
\]
Their pairing is zero, and both have dimension two. Thus
$\overline L_i=L_i^\perp$. Set $L=\bigoplus_iL_i$. Explicitly,
$\jmath(\ker C)=F\cap L$ and
$\jmath(\ker\Ct)=F\cap L^\perp$. We use this identification when
writing a quotient-face assignment as its local tuples $x_i$.
Consequently
\[
\dim(F\cap L^\perp)
=\dim(F+L)^\perp
=4N-\dim(F+L)
=\dim(F\cap L),
\]
and the reversed matrix has the same nullity $n$ as $C$.

We also need an identity for the inhomogeneous systems. If $Cx=Dt$
and $\Ct y=Dt'$, their local values have the form
\[
x_i=(u_i,u_i+v_i,v_i,v_i-\eps_it_i),\qquad
y_i=(p_i-\eps_it'_i,p_i,p_i+w_i,w_i).
\]
Since $x,y\in F$, their pairing is zero. Its $i$th summand is
$-t'_iu_i+t_iw_i$, giving
\begin{equation}
(t')^{\mathsf T}Rx=t^{\mathsf T}\Rt y.
\label{eq:mixed}
\end{equation}
In particular, if $C$ is invertible then so is $\Ct$, and
\begin{equation}
\Rt\Ct^{-1}D=(RC^{-1}D)^{\mathsf T}.
\label{eq:transpose}
\end{equation}

\section{The face relation}\label{sec:face-relation}\label{sec:relation}
We now allow singular face matrices. Define
\begin{align*}
S_C&=\{t:Dt\in\im C\},& O_C&=(R\ker C)^\perp,\\
S_{\Ct}&=\{t:Dt\in\im\Ct\},& O_{\Ct}&=(\Rt\ker\Ct)^\perp.
\end{align*}
Here perpendiculars in $\R^N$ use the ordinary dot product, whereas
perpendiculars in $V$ use \eqref{eq:pair}.

\begin{lemma}\label{lem:solvability}
One has $S_{\Ct}=O_C$ and $S_C=O_{\Ct}$.
\end{lemma}
\begin{proof}
A local particular solution of the reversed equations is
$y_i^0=(-\eps_it_i,0,0,0)$. There is a face-constant solution precisely
when the affine space $y^0+L^\perp$ meets $F$. This is equivalent to
$y^0\in F+L^\perp$, and therefore to
$y^0\perp (F+L^\perp)^\perp=F\cap L$. For
$k_i=(u_i,u_i+v_i,v_i,v_i)$, the corresponding local pairing is
$-t_iu_i$. Thus the solvability condition is exactly
$t\perp R\ker C$. The original particular solution
$x_i^0=(0,0,0,-\eps_it_i)$ gives the second equality in the same way:
its pairing with $(p_i,p_i,p_i+w_i,w_i)$ is $t_iw_i$.
\end{proof}

Put $W=S_C\cap S_{\Ct}=S_C\cap O_C$, and define a linear relation
between two copies of $\R^N$ by
\begin{equation}
\mathcal L_F=\{(t,s):Cx=Dt,\ \Ct y=Dt,\
s=Rx+\Rt y+\Delta t\text{ for some }x,y\in\R^{2N}\}.
\label{eq:relation}
\end{equation}
The projection $(t,s)\mapsto t$ has image $W$. Its kernel within
$\mathcal L_F$ is $\{0\}\times(R\ker C+\Rt\ker\Ct)$. By
Lemma~\ref{lem:solvability},
\[
R\ker C+\Rt\ker\Ct=(O_C\cap O_{\Ct})^\perp=W^\perp.
\]
The projection therefore gives a short exact sequence of vector spaces
\[
0\longrightarrow W^\perp\xrightarrow{s\mapsto(0,s)}
\mathcal L_F\xrightarrow{(t,s)\mapsto t}W\longrightarrow0,
\]
and in particular
\begin{equation}
\dim\mathcal L_F=\dim W^\perp+\dim W=N.
\label{eq:dim}
\end{equation}
At this stage no injectivity of the face parametrization has been used.

\section{Local cochains and the edge equations}\label{sec:local}
Fix $x,y,t,s$ as in \eqref{eq:relation}. For each quotient face $f$,
take its abstract characteristic triangle, with ordered corners
$a<b<c$, and assign potentials on these three corners:
\[
\phi_f(a)=0,\qquad\phi_f(b)=x_f,\qquad
\phi_f(c)=x_f-y_f.
\]
Their differences define an additive cochain on the three directed
edge occurrences of this triangle:
\[
h_f(ab)=x_f,\qquad h_f(ac)=x_f-y_f,\qquad h_f(bc)=-y_f.
\]
The potentials, and therefore the cochains, agree at corresponding
corners and edge occurrences of paired abstract faces because the
gluing preserves the order. The potentials are not functions on the
single quotient vertex. Nor do we yet assert that the cochain values
agree on distinct occurrences of a quotient edge, whether in the same
face or in different faces. Descent to quotient edges in the
zero-shape case is proved in Section~\ref{sec:injectivity}.
Figure~\ref{fig:facecochain} records the local convention.

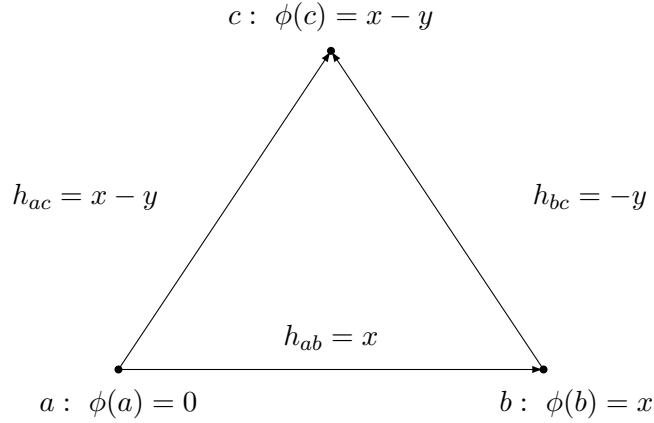
\begin{figure}[htbp]
\centering
\setlength{\unitlength}{1pt}
\begin{picture}(300,165)
\put(70,25){\vector(1,0){160}}
\put(70,25){\vector(2,3){80}}
\put(230,25){\vector(-2,3){80}}
\put(70,25){\circle*{3}}
\put(230,25){\circle*{3}}
\put(150,145){\circle*{3}}
\put(35,6){\makebox(70,12){$a:\ \phi(a)=0$}}
\put(195,6){\makebox(95,12){$b:\ \phi(b)=x$}}
\put(95,152){\makebox(110,12){$c:\ \phi(c)=x-y$}}
\put(130,31){\makebox(40,12){$h_{ab}=x$}}
\put(15,84){\makebox(85,12){$h_{ac}=x-y$}}
\put(205,84){\makebox(85,12){$h_{bc}=-y$}}
\end{picture}
\caption{The local cochain on an abstract face with ordered corners $a<b<c$.
Each arrow is oriented from the lower to the higher corner. Its value is the
difference of the endpoint potentials, so $h_{ab}+h_{bc}=h_{ac}$.}
\label{fig:facecochain}
\end{figure}

For a directed edge $jm$ of one tetrahedron, let $k,l$ be the two
remaining vertices. Denote the cochain on face $k$ by $h_k$ and set
\begin{equation}
q_{jm}=\eps\,\operatorname{sgn}(j,m,k,l)
\big(h_k(jm)-h_l(jm)\big).
\label{eq:q}
\end{equation}
Exchanging $j,m$ changes the permutation sign and both cochain signs;
exchanging $k,l$ changes the permutation sign and the order of the
difference. Thus $q_{jm}$ is an unoriented edge value, independent of
both choices. The local constraints are
\[
x_1=x_0+x_2,\quad x_2-x_3=\eps t,\qquad
y_2=y_1+y_3,\quad y_1-y_0=\eps t.
\]
They give
\begin{equation}
\begin{aligned}
q_{01}=q_{23}&=t,\\
q_{03}=q_{12}&=\eps(x_0+y_3),\\
q_{02}=q_{13}&=-t-\eps(x_0+y_3).
\end{aligned}
\label{eq:qtable}
\end{equation}
For example, $q_{01}=\eps(x_2-x_3)=t$ and
$q_{03}=\eps[(x_1-y_1)-(x_2-y_2)]=\eps(x_0+y_3)$.
Since $s=x_0+y_3+(\eps+1)t/2$, the positive-tetrahedron placement
gives $(q_z,q_{z'},q_{z''})=(t,s-t,-s)$. When $\eps=-1$, exchanging
the prime and double-prime edge pairs gives the same triple. This is
the placement fixed in \eqref{eq:shapes}.

For a fixed directed quotient edge, each local term in \eqref{eq:q}
is the incidence of an adjacent face times its cochain value.
Along a paired face, the two tetrahedra induce opposite face
orientations by \eqref{eq:facesign}, while the directed-edge cochain
value agrees. The two contributions therefore cancel. Summing over
the complete link of a quotient edge gives zero. This counts local
incidences separately and applies equally to self-gluings and repeated
edges. Hence the values in \eqref{eq:qtable} satisfy every homogeneous
edge equation.

\section{The longitude via a dual-spine cocycle}\label{sec:peripheral}
The edge cancellation does not by itself establish the peripheral
equation. To obtain that equation, we construct a cohomology class of
$M$ whose periods are the peripheral shape holonomies of the face
data. Its period on the preferred longitude is zero because that
curve is null-homologous in $M$.

Write $\phi_k(j)$ for the potential at corner $j$ on local face $k$.
The following auxiliary numbers express the difference of endpoint-potential
sums on two incident faces as a difference of face values:
\begin{equation}
\begin{aligned}
f_0&=0,&f_1&=-x_2-x_3,\\
f_2&=x_0-x_2-x_3+y_3,&f_3&=x_0-2x_3+y_3.
\end{aligned}
\label{eq:f}
\end{equation}
If faces $k,l$ meet in edge $jm$, then
\begin{equation}
\phi_k(j)+\phi_k(m)-\phi_l(j)-\phi_l(m)=f_l-f_k.
\label{eq:potential}
\end{equation}
Here is the full local check, with the pairs $k<l$ in increasing order:
\[
\begin{array}{c|c|l}
(k,l)&\{j,m\}&f_l-f_k\\\hline
(0,1)&\{2,3\}&-x_2-x_3\\
(0,2)&\{1,3\}&x_0-x_2-x_3+y_3\\
(0,3)&\{1,2\}&x_0-2x_3+y_3\\
(1,2)&\{0,3\}&x_0+y_3\\
(1,3)&\{0,2\}&x_0+y_3+\eps t\\
(2,3)&\{0,1\}&\eps t
\end{array}
\]
Substitution of the face potentials into the left side of
\eqref{eq:potential}, followed by the four local constraints from
Section~\ref{sec:local}, gives these entries.

\paragraph{The cusp sign.}
Consider a normal arc in the cusp triangle at vertex $j$. Suppose it
enters face $k$, leaves face $l$, and cuts off the corner corresponding
to edge $jm$. The truncation triangle has induced boundary orientation
$-\eps(-1)^j$ relative to the increasing order of its three vertices:
it is parallel to the face opposite $j$, with the opposite outward
normal. The arc goes from side $(m,l)$ to side $(m,k)$. With a corner
on the left counted positively, as fixed in Section~1.1, its signed
turn is therefore
\[
[-\eps(-1)^j]\,\operatorname{sgn}
(m,l,k\text{ relative to the remaining increasing triple})
=\eps\operatorname{sgn}(j,m,k,l).
\]
Multiplying this sign by $q_{jm}$ in \eqref{eq:q} gives the signed
shape contribution of the arc. Equation~\eqref{eq:potential} rewrites
that contribution as
\begin{equation}
h_k(jm)-h_l(jm)
=[f_l+2\phi_l(j)]-[f_k+2\phi_k(j)].
\label{eq:arc}
\end{equation}
If the peripheral holonomy convention uses the opposite global turning
sign, all these periods change sign together; their vanishing is
unchanged.

\paragraph{Closure on the spine.}
The dual two-skeleton of the ideal triangulation is a spine of $M$.
On a directed dual edge crossing from local face $(i,k)$ to its paired
face $(i',l)$, put
\[
a\big((i,k)\longrightarrow(i',l)\big)=f_{i,k}-f_{i',l}.
\]
The two local face occurrences remain distinct when $i=i'$, so this
also defines $a$ on a dual loop. Reversing the dual edge negates its
value.

To prove that $a$ is closed, traverse the boundary of a dual two-cell,
which encircles one quotient edge. In successive tetrahedral wedges,
let $k_r$ and $l_r$ be the incoming and outgoing faces. The period of
$a$ around the cell is
\[
\sum_r(f_{i_r,l_r}-f_{i_r,k_r}).
\]
By \eqref{eq:potential}, each term is the incoming sum of endpoint
potentials minus the outgoing sum. The outgoing face of one wedge is
paired with the incoming face of the next, and the two endpoint
potentials agree. Every term cancels. The argument counts occurrences,
so it also applies to loop edges and self-gluings. Therefore $a$
represents a class in $H^1(M;\R)$.

\paragraph{Peripheral periods.}
On a closed normal peripheral curve, the $\phi$ terms in
\eqref{eq:arc} cancel across successive paired faces. Its shape
holonomy is consequently $\sum_r(f_{i_r,l_r}-f_{i_r,k_r})$. Pushing the
curve into the dual spine follows exactly those face crossings, so
this sum is its period under $a$. The push preserves the homotopy
class: within each truncated tetrahedron one replaces the local arc
by a path through its center and the centers of the crossed faces.
Since $[\lambda]=0$ in $H_1(M;\R)$, the longitude period is zero.

Combined with the edge equations from Section~\ref{sec:local}, this
proves $\mathcal L_F\subseteq\ker(A\mid-B)$. Lemma~\ref{lem:nzfacts}
gives dimension $N$ for the right side, and \eqref{eq:dim} gives the
same dimension for the left side. Thus
\begin{equation}
\mathcal L_F=\ker(A\mid-B).
\label{eq:duality}
\end{equation}
The block-column matrix $\binom{B^{\mathsf T}}{A^{\mathsf T}}$ has rank
$N$ and, by $AB^{\mathsf T}=BA^{\mathsf T}$, its image lies in this
kernel. Equality of dimensions gives
\[
\ker(A\mid-B)=\im\begin{pmatrix}B^{\mathsf T}\\A^{\mathsf T}\end{pmatrix},
\qquad W=\im B^{\mathsf T}.
\]

\section{Unique face lifts and the nonsingular identity}\label{sec:factor-two}\label{sec:injectivity}
We have identified the face and edge relations without assuming that
face lifts are unique. Their uniqueness is the topological step
responsible for nullity doubling.

Consider the map
\[
P:\ker C\oplus\ker\Ct\longrightarrow\R^N,
\qquad P(x,y)=Rx+\Rt y.
\]
Suppose $P(x,y)=0$. Then $t=s=0$ in \eqref{eq:relation}, so every
$q_{jm}$ in \eqref{eq:qtable} is zero. Equation~\eqref{eq:q} says that
the cochains on the two faces meeting each tetrahedron edge agree
along that edge. They also agree across paired faces. Following the
connected link of each quotient edge gives a well-defined cochain on
the edges of the ideal compactification $X=M/\partial M$. On every
triangle it is a difference of vertex potentials, so it is a cellular
$1$-cocycle.

Conversely, let $h\in Z^1(X;\R)$. On the abstract triangle of face $f$
with ordered corners $a<b<c$, put $x_f=h(ab)$ and $y_f=-h(bc)$,
evaluating $h$ on the quotient edges represented by those occurrences.
These are quotient-face
values, since $h$ is global and the gluing preserves the order. In a
tetrahedron they read
\begin{align*}
(x_0,x_1,x_2,x_3)&=(h_{12},h_{02},h_{01},h_{01}),\\
(y_0,y_1,y_2,y_3)&=(-h_{23},-h_{23},-h_{13},-h_{12}).
\end{align*}
The triangle equations $h_{ab}+h_{bc}=h_{ac}$ give
$Cx=\Ct y=0$, and $Rx+\Rt y=0$ follows directly. The reconstructed
face potentials recover $h$ because $x_f-y_f=h(ac)$. These two
constructions are inverse, hence
\[
\ker P\cong Z^1(X;\R).
\]

The complex $X$ has one vertex, so the coboundary
$C^0(X;\R)\to C^1(X;\R)$ is zero. Therefore $Z^1(X;\R)=H^1(X;\R)$.
The quotient description and Poincar\'e--Lefschetz duality give
\[
H^1(X;\R)\cong H^1(M,\partial M;\R)
\cong H_2(M;\R)=0.
\]
The last equality holds for a knot exterior in $S^3$. Thus $P$ is
injective. Its domain has dimension $2n$, and its image is $W^\perp$
by Section~\ref{sec:face-relation}. Since $W=\im B^{\mathsf T}$,
\[
\nullity B=\dim W^\perp=2n.
\]
This proves the first part of Theorem~\ref{thm:main}. It also proves
uniqueness of every face lift in \eqref{eq:duality}, because the
difference of two lifts belongs to $\ker P$.

If $n=0$, the face relation is the graph of
\[
s=(RC^{-1}D+\Rt\Ct^{-1}D+\Delta)t.
\]
Nullity doubling makes $B$ invertible. Comparing this graph with
\eqref{eq:duality} and using \eqref{eq:transpose} gives
\[
B^{-1}A=RC^{-1}D+(RC^{-1}D)^{\mathsf T}+\Delta,
\]
which proves the second part of Theorem~\ref{thm:main}.

\section{Singular constraint spaces and forms}\label{sec:singular}
Choose invertible rational matrices $E_1,E_2$ such that
\[
E_1CE_2=\diag(I_r,0_n),\qquad r=2N-n.
\]
For a matrix partitioned accordingly, a subscript $r$ denotes the
first $r$ columns or rows and a subscript $n$ the remaining ones.
Define
\begin{align*}
K_F&=\begin{pmatrix}(RE_2)_n^{\mathsf T}\\(E_1D)_n\end{pmatrix},\\
H_F&=(RE_2)_r(E_1D)_r,\qquad
G_F=H_F+H_F^{\mathsf T}+\Delta.
\end{align*}
The last $n$ columns of $E_2$ form a basis of $\ker C$, so the first
block of $K_F$ imposes $t\perp R\ker C$. The second block imposes
solvability of $Cx=Dt$. Thus
\[
\ker K_F=W,\qquad\rank K_F=2n.
\]
The rank follows from $\operatorname{codim}W=2n$, already proved;
there is no assumption that these two sets of constraints are
independent before that argument.

For $t\in W$, a particular solution is
$x_t=(E_2)_r(E_1D)_r t$. If $(t,s)\in\mathcal L_F$, choose a
reversed solution $\Ct y=Dt$. For $v\in W$, changes of an original
solution lie in $\ker C$ and disappear after multiplication by
$v^{\mathsf T}R$. Applying \eqref{eq:mixed} to $Cx_v=Dv$ and
$\Ct y=Dt$ gives
\[
v^{\mathsf T}\Rt y=t^{\mathsf T}Rx_v.
\]
Consequently
\[
v^{\mathsf T}s
=v^{\mathsf T}Rx_t+t^{\mathsf T}Rx_v+v^{\mathsf T}\Delta t
=v^{\mathsf T}G_Ft.
\]
Conversely, for fixed $t\in W$ every $s$ in the fiber differs from
any one such value by $W^\perp$. Hence the entire relation, including
its off-diagonal bilinear information, is
\[
\mathcal L_F=
\{(t,s):t\in W,\ s-G_Ft\in W^\perp\}.
\]

To compare this with the row conditions in \cite[Definition~1.1]{Wong},
row-reduce $(B\mid A)$ and write its $N$ independent rows as
\[
\begin{pmatrix}B_0&A_0\\0&K_{NZ}\end{pmatrix},
\]
where $B_0$ has $N-2n$ rows and $K_{NZ}$ has $2n$ rows. The equation
$At=Bs$ is equivalent to $K_{NZ}t=0$ and $A_0t=B_0s$. Since $B_0$
has full row rank, the second equation can be solved for $s$ for
every $t$ satisfying the first. Equation~\eqref{eq:duality} therefore
gives
\[
\ker K_{NZ}=W=\ker K_F.
\]
The matrices $K_{NZ}$ and $K_F$ both have $2n$ independent rows, so
they are row equivalent.

Furthermore, $\operatorname{row}B_0=\operatorname{row}B=W$. Thus
$B_0$ annihilates $W^\perp$. The preceding description of
$\mathcal L_F$ now gives $A_0t=B_0G_Ft$ for every $t\in W$.
Each row of $A_0-B_0G_F$ lies in $W^\perp=\operatorname{row}K_{NZ}$.
There is consequently a rational matrix $Q$ with
$A_0-B_0G_F=QK_{NZ}$, and the explicit invertible row operation
\[
\begin{pmatrix}I&-Q\\0&I\end{pmatrix}
\begin{pmatrix}B_0&A_0\\0&K_{NZ}\end{pmatrix}
=
\begin{pmatrix}B_0&B_0G_F\\0&K_{NZ}\end{pmatrix}
\]
proves the required quadratic compatibility. These arguments establish
the third part of Theorem~\ref{thm:main} for every rational choice of
$E_1,E_2$.

\section{Strict angles and positive gluing rank}\label{sec:geometric}
A strict angle structure assigns an angle in $(0,\pi)$ to each pair
of opposite tetrahedron edges, with sum $\pi$ in each tetrahedron
and sum $2\pi$ around each quotient edge. In particular, each
quotient edge has degree at least three, where degree counts local
edge occurrences. The following combinatorial lemma provides the
strict rank inequality that does not follow from nullity doubling
alone.

\begin{lemma}\label{lem:nozero}
If the ordered knot triangulation admits a strict angle structure,
its full edge block $B_{\rm edge}$ is nonzero. Consequently
$\rank B>0$.
\end{lemma}
\begin{proof}
Suppose instead that $B_{\rm edge}=0$. Orient local edges and faces
by increasing vertex order; these orientations agree across paired
faces. In the cellular chain complex of $X=M/\partial M$, write
$e_{ab}$ for the quotient-edge generator represented by local edge
$ab$ with $a<b$, and $f_k$ for a local face viewed as a quotient-face
generator. The local boundaries are
\[
\begin{aligned}
\partial_2 f_0&=e_{12}-e_{13}+e_{23},&
\partial_2 f_1&=e_{02}-e_{03}+e_{23},\\
\partial_2 f_2&=e_{01}-e_{03}+e_{13},&
\partial_2 f_3&=e_{01}-e_{02}+e_{12}.
\end{aligned}
\]
A zero tetrahedron column in $B_{\rm edge}$ means
\begin{equation}
e_{03}+e_{12}=e_{02}+e_{13}
\label{eq:zerocolumn}
\end{equation}
in the free vector space on quotient edges. The tetrahedron sign
only exchanges the two sides. Thus both $f_0-f_1$ and $f_2-f_3$
belong to $\ker\partial_2$ for each tetrahedron. Since
$H_1(X;\Q)=0$ and $X$ has one vertex, $N$ edges and $2N$ faces,
\[
\rank\partial_2=N,\qquad\dim\ker\partial_2=N.
\]

Construct a multigraph whose vertices are the $2N$ quotient faces.
Tetrahedron $i$ contributes an edge $P_i$ joining $f_{i0}$ to
$f_{i1}$ and an edge $Q_i$ joining $f_{i2}$ to $f_{i3}$. Every
quotient face has two local occurrences, each used once, so the
graph is $2$-regular, counting a loop twice. If it has $k$
components, its incidence vectors span a space of dimension $2N-k$.
These vectors are the displayed face differences in $\ker\partial_2$,
so $2N-k\leq N$, or $k\geq N$.

A loop $P_i$ glues faces $0$ and $1$ of tetrahedron $i$ to each
other. The ordered gluing fixes their common edge $23$ pointwise.
Both sides of that local edge link are then paired to each other,
giving a quotient edge of degree one. A loop $Q_i$ similarly gives
degree one at edge $01$. Strict angles exclude these possibilities.
Every graph component therefore has at least two vertices. Together
with $k\geq N$, this forces exactly $N$ components, each with two
vertices and two parallel edges.

If a component consists of $P_i,P_j$, the two face gluings pair
faces $0,1$ of tetrahedron $i$ with faces $0,1$ of tetrahedron $j$.
Their common edges $23_i,23_j$ form a closed link cycle of length
two: both adjacent faces of both edge occurrences are already paired
within the component. Hence the resulting quotient edge has degree
two. A component of two $Q$ edges has the same property at edge
$01$. Strict angles exclude these cases as well. Every component
must therefore consist of one $P_i$ and one $Q_j$. In particular,
each tetrahedron appears exactly once as an index $i$ and once as
an index $j$ among these mixed components.

There are two possible ordered face pairings in a mixed component.
For $f_{i0}\sim f_{j2}$ and $f_{i1}\sim f_{j3}$, the induced edge
identifications are
\[
\begin{array}{lll}
e_{12}(i)=e_{01}(j),&e_{13}(i)=e_{03}(j),&e_{23}(i)=e_{13}(j),\\
e_{02}(i)=e_{01}(j),&e_{03}(i)=e_{02}(j),&e_{23}(i)=e_{12}(j).
\end{array}
\]
For $f_{i0}\sim f_{j3}$ and $f_{i1}\sim f_{j2}$, they are
\[
\begin{array}{lll}
e_{12}(i)=e_{01}(j),&e_{13}(i)=e_{02}(j),&e_{23}(i)=e_{12}(j),\\
e_{02}(i)=e_{01}(j),&e_{03}(i)=e_{03}(j),&e_{23}(i)=e_{13}(j).
\end{array}
\]
Both tables imply $e_{12}(i)=e_{02}(i)$ and
$e_{12}(j)=e_{13}(j)$. Since every tetrahedron occurs in each role,
$e_{02}=e_{12}=e_{13}$ in every tetrahedron. Equation~\eqref{eq:zerocolumn}
then forces $e_{03}$ to be the same quotient edge. This is equality
of basis vectors in $C_1(X;\Q)$, not merely equality of homology
classes. Denote the common edge by $c_i$.

Either pairing table gives $c_i=c_j$, $e_{01}(j)=c_i$, and
$e_{23}(i)=c_j$. Every face gluing belongs to a mixed component.
The tetrahedron adjacency graph is connected, so all local edges in
the triangulation represent a single quotient edge. Since their
number is $N$, it follows that $N=1$.

For one tetrahedron, \eqref{eq:facesign} allows only opposite-parity
faces to be paired. The two possibilities are $(01)(23)$ and
$(03)(12)$. The first glues faces $0,1$ together, already shown to
give degree one. In the second, the ordered gluing of face
$1=(023)$ to face $2=(013)$ fixes edge $03$ pointwise, again giving
degree one. This contradicts strict angles and proves
$B_{\rm edge}\ne0$.

Finally, the sum of all edge rows of $B_{\rm edge}$ is zero: each
tetrahedron has two occurrences of each shape type. If deleting one
row left a zero matrix, that row would be zero as well. Thus the
retained edge rows, and hence the longitude matrix $B$, are nonzero.
\end{proof}

For clarity, in our notation \cite[Definition~1.1]{Wong} requires
the following four conditions for generalized FAMED with respect
to the longitude:
\begin{enumerate}
\item the strict angle space is nonempty;
\item $\nullity B=2n<N$, where $n=\nullity C$;
\item $K_{NZ}$ and $K_F$ from Section~\ref{sec:singular} are row equivalent;
\item the two block matrices
\[
\begin{pmatrix}B_0&A_0\\0&K_{NZ}\end{pmatrix}
\quad\hbox{and}\quad
\begin{pmatrix}B_0&B_0G_F\\0&K_{NZ}\end{pmatrix}
\]
are row equivalent.
\end{enumerate}
Here $(B_0,A_0,K_{NZ})$ is obtained from the reduced row echelon
form of $(B\mid A)$, as in that definition, and our $G_F$ is its
matrix $\mathscr G$.

\begin{proof}[Proof of Theorem~\ref{thm:generalized}]
The first condition is the angle-structure hypothesis.
Theorem~\ref{thm:main} and Lemma~\ref{lem:nozero} give
\[\nullity B=2n=N-\rank B<N,\]
which is the second condition. The two row equivalences were proved
in Section~\ref{sec:singular}, including an explicit row operation
for the fourth condition.
\end{proof}

The rank lemma uses only the absence of degree-one and degree-two
edges, together with $H_1(X;\Q)=0$. It asserts that the full edge
block is nonzero; it does not assert that every tetrahedron column
is nonzero.

\begin{corollary}[Wong's upper bound without a FAMED hypothesis]\label{cor:upperbound}
Let $\mathscr A_{\mathcal T}$ be the nonempty strict angle space of
an ordered ideal triangulation of a hyperbolic knot complement in
$S^3$. For $\alpha\in\mathscr A_{\mathcal T}$, let
$\mathscr Z_\hbar(\mathcal T,\alpha)$ be its Teichm\"uller TQFT
partition function, and let $\mathrm H_\lambda^{\R}$ denote angular
longitude holonomy, with the conventions of \cite{Wong}. Then
\[
\limsup_{\hbar\to0^+}2\pi\hbar\log|\mathscr Z_\hbar(\mathcal T,\alpha)|
\leq-\sup_{\substack{\alpha'\in\mathscr A_{\mathcal T}\\
\mathrm H_\lambda^{\R}(\alpha')=\mathrm H_\lambda^{\R}(\alpha)}}
\operatorname{Vol}(\alpha').
\]
Here $\operatorname{Vol}$ is the sum of the ideal-tetrahedron volumes
determined by the angles.
\end{corollary}
\begin{proof}
Theorem~\ref{thm:generalized} supplies the generalized FAMED
hypothesis in the first clause of \cite[Theorem~1.11]{Wong}. That
analytic theorem gives the inequality. No further estimate is needed.
\end{proof}

\begin{corollary}[The partition-function volume limit]\label{cor:volumelimit}
Suppose $\mathcal T$ is an ordered complete geometric ideal
triangulation of a hyperbolic knot complement $M=S^3\smallsetminus K$,
and let $\alpha_{\rm geom}$ be its geometric angle structure.
For every strict angle structure $\alpha$ with
$\mathrm H_\lambda^{\R}(\alpha)=
\mathrm H_\lambda^{\R}(\alpha_{\rm geom})$, one has
\[
\lim_{\hbar\to0^+}2\pi\hbar
\log|\mathscr Z_\hbar(\mathcal T,\alpha)|=-\operatorname{Vol}(M).
\]
In particular, the conclusion holds for $\alpha=\alpha_{\rm geom}$.
It allows singular face matrices.
\end{corollary}
\begin{proof}
Theorem~\ref{thm:generalized} supplies generalized FAMED. Let
$\mathbf z_{\rm geom}$ denote the complete shapes, all with positive
imaginary parts. They satisfy the edge equations on the geometric
logarithmic branch. Completeness makes the real part of their
logarithmic longitude holonomy zero, and the imaginary part is the
angular holonomy of \cite[Section~2.2]{Wong}. Any constants introduced
by simplifying the corner expressions are retained in both holonomies. Thus
\[
\mathrm H_\lambda^{\mathbb C}(\mathbf z_{\rm geom})
=i\mathrm H_\lambda^{\R}(\alpha_{\rm geom})
=i\mathrm H_\lambda^{\R}(\alpha).
\]
The additional shape-solution hypothesis of the second clause of
\cite[Theorem~1.11]{Wong} is therefore satisfied. Its logarithmic
limit gives the conclusion, since the sum of these tetrahedral
volumes is $\operatorname{Vol}(M)$.
\end{proof}
The angular holonomies in this corollary use the same normal
peripheral representative and logarithmic convention. When the lift
is normalized to have zero holonomy at the complete structure, its
hypothesis reads $\mathrm H_\lambda^{\R}(\alpha)=0$. Exponentiated
completeness alone does not select a logarithmic lift.

These consequences use Wong's analytic theorem. They do not establish
the existence of a complete geometric triangulation for every knot,
or the additional longitude--meridian conditions in
\cite[Definition~1.5]{Wong} needed for its Jones-function assertions.

\section{An affine face formulation of the gluing equations}\label{sec:affine}
We now express the lifted gluing equations in face coordinates. Related correspondences between logarithmic gluing equations and critical-point equations appear in \cite[Proposition~3.4]{BW} and \cite[Lemma~3.6 and Proposition~3.9]{Wong}. The construction here uses the face-cochain isomorphism proved above; it does not require a potential function or invertibility of $C$.

Give each quotient face two complex coordinates $x_f,y_f$. Define the complex linear space $\mathcal F$ by
\begin{equation}
\begin{aligned}
(X_0-X_1+X_2)x&=0,\\
(X_3-X_2+X_1)y&=0,\\
(X_2-X_3)x&=(X_1-X_0)y.
\end{aligned}\label{eq:affineface}
\end{equation}
Set
\[
t=E(X_2-X_3)x,\qquad s=Rx+\Rt y+\Delta t.
\]
These are the simultaneous face equations with $t$ eliminated. By the duality and injectivity arguments in Sections~\ref{sec:peripheral}--\ref{sec:injectivity}, the map
\[
\Phi:\mathcal F\longrightarrow\ker_{\mathbb C}(A\mid-B),\qquad (x,y)\longmapsto(t,s)
\]
is a complex linear isomorphism. Thus $\dim_{\mathbb C}\mathcal F=N$, including when $C$ is singular.

\subsection{The lifted solution set}
Fix $u\in\mathbb C$ and $\kappa\in\Z^N$. The first $N-1$ entries of $\kappa$ are even, and are $2$ for the usual geometric edge equations. Fix also the representative and additive constant in the longitude formula: if
\[
H_\lambda=i\pi c_\lambda+(GZ+G'Z'+G''Z'')_N,
\qquad c_\lambda\in\Z,
\]
then take $\kappa_N=-c_\lambda$, so that the last equation below prescribes $H_\lambda=u$.

Let $\mathcal S_{\kappa,u}$ consist of triples $(Z,Z',Z'')\in\mathbb C^{3N}$ such that, for every $i$,
\[
z_i=e^{Z_i}\notin\{0,1\},\qquad
e^{Z'_i}=\frac{1}{1-z_i},\qquad
e^{Z''_i}=1-\frac{1}{z_i},
\]
and
\[
Z+Z'+Z''=i\pi\mathbf1,\qquad
GZ+G'Z'+G''Z''=i\pi\kappa+u e_N.
\]
Here $e_N$ is the last standard basis vector of $\mathbb C^N$. These are normalized logarithmic lifts: the variables are not required to be principal logarithms. No positivity of the shapes or existence of a solution is assumed.

Put $\eta=\kappa-G'\mathbf1$ and choose rational vectors satisfying
\[
Ap_0-Br_0=\eta,\qquad Aq-Br=e_N.
\]
Such choices exist by the full row rank of $(A\mid B)$. For $(x,y)\in\mathcal F$ with $\Phi(x,y)=(t,s)$, define
\begin{equation}
Z=i\pi p_0+t+u q,\qquad S=i\pi r_0+s+u r.
\label{eq:affineZS}
\end{equation}

\begin{proposition}\label{prop:affine}
For each fixed $\kappa$ and $u$ as above, the subset of $\mathcal F$ satisfying
\begin{equation}
\exp(Z_i)\big(1-\exp(-S_i)\big)=1,\qquad i=1,\ldots,N,
\label{eq:nonlinearface}
\end{equation}
is in bijection with $\mathcal S_{\kappa,u}$. The bijection sends $(x,y)$ to
\[
(Z,Z',Z'')=(Z,\ i\pi\mathbf1-Z+S,\ -S).
\]
\end{proposition}
\begin{proof}
The affine substitutions give $AZ-BS=i\pi\eta+u e_N$. Hence the displayed triple has tetrahedral sum $i\pi\mathbf1$ and satisfies
\[
GZ+G'Z'+G''Z''=AZ-BS+i\pi G'\mathbf1
=i\pi\kappa+u e_N.
\]
Equation~\eqref{eq:nonlinearface} gives $z_i(1-z_i'')=1$ with $z_i=e^{Z_i}$ and $z_i''=e^{-S_i}$. Both exponentials are nonzero, so $z_i\ne1$, $z_i''=1-1/z_i$, and
\[
e^{Z_i'}=e^{i\pi-Z_i+S_i}=\frac{1}{1-z_i}.
\]
Thus the triple belongs to $\mathcal S_{\kappa,u}$.

Conversely, let $(Z,Z',Z'')\in\mathcal S_{\kappa,u}$ and put $S=-Z''$. The tetrahedral sums give $Z'=i\pi\mathbf1-Z+S$. Subtracting the two particular solutions in \eqref{eq:affineZS} produces
\[
(t,s)=(Z-i\pi p_0-uq,\ S-i\pi r_0-ur)\in\ker_{\mathbb C}(A\mid-B).
\]
The isomorphism $\Phi$ gives a unique $(x,y)\in\mathcal F$. The shape relation for $e^{Z''}$ gives \eqref{eq:nonlinearface}, establishing the inverse.
\end{proof}

\subsection{Auxiliary choices and logarithmic lifts}
The choices of $p_0,r_0,q,r$ affect only the affine origin. If they are replaced by $\widetilde p_0,\widetilde r_0,\widetilde q,\widetilde r$ satisfying the same equations, then
\[
\delta_u=i\pi(p_0-\widetilde p_0,r_0-\widetilde r_0)
+u(q-\widetilde q,r-\widetilde r)
\in\ker_{\mathbb C}(A\mid-B).
\]
Translation by $\Phi^{-1}(\delta_u)$ identifies the two face solution sets and leaves the resulting triple $(Z,Z',Z'')$ unchanged. This assertion keeps $\kappa$ and $u$ fixed.

The omitted edge equation has forced logarithmic constant
\[
i\pi\left(2N-\sum_{j=1}^{N-1}\kappa_j\right),
\]
because each tetrahedron contributes twice its logarithmic sum to the sum of all edge equations. It is $2\pi i$ when every retained edge constant is $2\pi i$.

A fixed $\kappa$ specifies lifted linear constraints, not a unique logarithmic branch. For example, whenever $a,b\in\Z^N$ satisfy $Aa-Bb=0$, replacing $(Z,S)$ by $(Z+2\pi i a,S+2\pi i b)$ preserves $\kappa$, $u$, and all exponentiated shapes. The proposition is a bijection of the specified lifted solution sets; it does not assert a bijection after forgetting logarithmic lifts.

At prescribed $u$, all unknowns in \eqref{eq:affineface} and \eqref{eq:nonlinearface} are attached to quotient faces. Positive shapes can be selected by imposing $\operatorname{Im}(e^{Z_i})>0$; a prescribed principal-logarithm domain requires its corresponding restrictions on $Z,Z',Z''$ as well. To impose both peripheral completeness equations, substitute \eqref{eq:affineZS} into the meridian equation. Setting $u=0$ alone is not a global assertion of meridian completeness.

\subsection{Figure-eight constants}
For the ordering in Example~\ref{ex:figureeight}, the gluing data in \cite[Figure~4]{BGW} and \cite[Section~2]{RIMS} give
\[
A=\begin{pmatrix}1&-2\\0&4\end{pmatrix},\qquad
B=\begin{pmatrix}-1&-1\\0&2\end{pmatrix},\qquad
\kappa=\begin{pmatrix}2\\-2\end{pmatrix},\qquad
\eta=\begin{pmatrix}-1\\2\end{pmatrix}.
\]
The longitude constant here is $c_\lambda=2$. One may choose
\[
p_0=\begin{pmatrix}1/3\\1/3\end{pmatrix},\qquad
r_0=\begin{pmatrix}-1/3\\-1/3\end{pmatrix},\qquad
q=\begin{pmatrix}1/2\\1/4\end{pmatrix},\qquad r=0.
\]
Direct multiplication gives $Ap_0-Br_0=\eta$ and $Aq-Br=e_2$. At $x=y=0$ and $u=0$, all three logarithms in each tetrahedron equal $i\pi/3$, recovering the regular shapes $z_i=z_i'=z_i''=e^{i\pi/3}$. This example checks the additive longitude constant as well as the sign of $S=-Z''$.

\section{Integral reversal symmetry}\label{sec:reversal}

\begin{theorem}\label{thm:integral-reversal}
For every ordered oriented face pairing, the integral matrices $C$ and
$\Ct$ have the same Smith normal form.
\end{theorem}
\begin{proof}
We first make the local change of coordinates explicit. For four local
face values $x_0,x_1,x_2,x_3$, write
\[
 a=x_0-x_1+x_2,\qquad b=x_2-x_3,\qquad
 c=x_3-x_2+x_1,\qquad d=x_1-x_0.
\]
The first pair is the original constraint and the second pair is the
reversed constraint. The inverse is integral:
\[
 x_0=b+c-d,\qquad x_1=b+c,\qquad
 x_2=a+d,\qquad x_3=a+d-b.
\]
Consequently this change of coordinates is unimodular. Put
\[
 h=\begin{pmatrix}0&1\\1&-1\end{pmatrix},\qquad
 H=\begin{pmatrix}0&E\\E&-E\end{pmatrix}.
\]
A direct expansion gives the local bilinear identity
\[
 \eps\sum_{k=0}^3(-1)^k x_k x'_k
 =\eps(a,b)h(a',b')^{\mathsf T}
  -\eps(c,d)h(c',d')^{\mathsf T}.
\]
Thus, after collecting the original and reversed constraints into separate
blocks, the global pairing has matrix $\diag(H,-H)$.

Let $\mathscr F_{\Z}\subset\Z^{4N}$ be the lattice of local assignments
constant across each paired face. The two diagonal entries of the original
pairing corresponding to a paired face are $\sigma,-\sigma$, with
$\sigma\in\{1,-1\}$. On this pair, pairing with a constant assignment
$(r,r)$ is $\sigma r(u-v)$. Hence the pairing map onto
$\mathscr F_{\Z}^{*}$ is surjective, and its kernel is exactly
$\mathscr F_{\Z}$. Applying the unimodular coordinate change gives an
exact sequence
\[
 0\longrightarrow\Z^{2N}
 \xrightarrow{\binom C{\Ct}}\Z^{4N}
 \xrightarrow{(C^{\mathsf T}H\ \ -\Ct^{\mathsf T}H)}\Z^{2N}
 \longrightarrow0.
\]
In particular,
\[
 C^{\mathsf T}HC=\Ct^{\mathsf T}H\Ct.
\]
Define
\[
 \Phi:\coker C\longrightarrow\coker(\Ct^{\mathsf T}H),
 \qquad [y]\longmapsto[C^{\mathsf T}Hy].
\]
The displayed identity makes $\Phi$ well defined. Surjectivity follows
from the last map in the exact sequence: every vector is congruent to a
vector $C^{\mathsf T}Hy$ modulo $\Ct^{\mathsf T}H\Z^{2N}$. If
$C^{\mathsf T}Hy=\Ct^{\mathsf T}Hz$, exactness gives
$(y,z)=(Cx,\Ct x)$ for some integral $x$, so $[y]=0$. Thus $\Phi$ is an
isomorphism. Since $H$ is unimodular and transposition preserves Smith
normal form, the theorem follows.
\end{proof}

Reversal preserves each opposite-edge pair and the tetrahedron orientation
sign. Retaining the quotient-edge labels and the same normal representative
of the peripheral curve therefore retains $A$ and $B$. Together with
\eqref{eq:transpose} and Section~\ref{sec:singular}, this also proves
reversal invariance of the constrained face form. Theorem~\ref{thm:integral-reversal}
strengthens the numerical reversal observation in Section~3.3 of
\cite{BGW} to an equality of integral Smith normal forms.

\section{Integral lifts, the cokernel sequence, and determinants}\label{sec:integral}

We now prove Theorems~\ref{thm:integral} and~\ref{thm:det}. All kernels
and cokernels in this section are taken over $\Z$ unless another
coefficient ring is displayed. Put
\[
 K=\begin{pmatrix}C&0&-D\\0&\Ct&-D\end{pmatrix},\qquad
 F_0=\begin{pmatrix}C&0\\0&\Ct\\R&\Rt\end{pmatrix}.
\]
The ideal compactification $X=M/\partial M$ has one vertex and
\[
 H_1(X;\Z)\cong H_1(M,\partial M;\Z)
 \cong H^2(M;\Z)=0.
\]
Here the quotient identification uses connected boundary, the second
isomorphism is Poincar\'e--Lefschetz duality, and the last equality follows
from the integral homology of a knot exterior.

\subsection{The cokernel of the simultaneous constraints}

\begin{lemma}\label{lem:simultaneous-presentation}
There is an integral presentation
\begin{equation}
 \coker K\cong H_1(X;\Z).
 \label{eq:kcoker}
\end{equation}
\end{lemma}
\begin{proof}
Use four row generators $p_i,q_i,p'_i,q'_i$ for each tetrahedron. The last
$N$ columns of $K$ impose $q'_i=-q_i$. Make the unimodular substitutions
$p_i=\eps_i a_i$, $p'_i=\eps_i b_i$, and $q_i=\eps_i c_i$.
For a local face $k$, factor its induced sign $\eps_i(-1)^k$ out of
each of its two contributions. The remaining expressions are
\[
\begin{array}{c|rrrr}
 k&0&1&2&3\\\hline
 F_k&a&a&a+c&c\\
 G_k&c&c-b&-b&-b.
\end{array}
\]
The two induced face signs in a pairing are opposite by
\eqref{eq:facesign}. Thus the column relations identify the two $F$
values and the two $G$ values across each paired face.

Assign formal potentials to the vertices of one tetrahedron by
\[
 (v_0,v_1,v_2,v_3)=(0,-b,c-b,a+c-b).
\]
On each increasingly ordered face, $G_k$ is the difference along its
lower consecutive edge and $F_k$ is the difference along its upper
consecutive edge. Their sum is the difference along its remaining edge.
Conversely, any additive assignment to the six oriented edges of one
tetrahedron arises uniquely this way: take $a=h_{23}$, $b=-h_{01}$ and
$c=h_{12}$, and use the triangle relations for the other three edges.

The presentation therefore consists of local directed-edge generators,
triangle-boundary relations, and identifications of matching edges across
paired faces. It is $C_1(X;\Z)/\im\partial_2$. Because $X$ has one
vertex, $\partial_1=0$, which proves the assertion. Repeated incidences
and self-pairings merely identify some of these generators; the same
presentation remains valid.
\end{proof}

Define
\[
 \Gamma=\{t\in\Z^N:Dt\in C\Z^{2N}
                   \text{ and }Dt\in\Ct\Z^{2N}\}.
\]
The homomorphism
\[
 \alpha:\Z^N\longrightarrow\coker C\oplus\coker\Ct,
 \qquad t\longmapsto([Dt]_C,[Dt]_{\Ct})
\]
has kernel $\Gamma$. Its cokernel is $\coker K$, which vanishes by
Lemma~\ref{lem:simultaneous-presentation}. Consequently
\begin{equation}
 \Z^N/\Gamma\cong\coker C\oplus\coker\Ct.
 \label{eq:facequotient}
\end{equation}
This statement includes singular face matrices.

\subsection{Saturation of the face lifts}

\begin{lemma}\label{lem:primitive-lifts}
The map $F_0$ is injective, and its image is primitive in $\Z^{5N}$.
Equivalently, every integral vector in its rational image has an integral
preimage.
\end{lemma}
\begin{proof}
We identify its kernel over an arbitrary field $k$, including characteristic
two. An element in $\ker(F_0\otimes k)$ satisfies locally
\[
 x_1=x_0+x_2,\qquad x_3=x_2,\qquad
 y_2=y_1+y_3,\qquad y_0=y_1,\qquad y_3=-x_0.
\]
Give each ordered face the edge values $(x,x-y,-y)$ on its lower,
diagonal, and upper edges. The two incident faces give the following
values on the six local tetrahedron edges:
\[
\begin{array}{c|cc}
 \text{edge}&\text{first value}&\text{second value}\\\hline
 01&x_2&x_3\\
 02&x_1&x_3-y_3\\
 03&x_1-y_1&x_2-y_2\\
 12&x_0&-y_3\\
 13&x_0-y_0&-y_2\\
 23&-y_0&-y_1
\end{array}
\]
The displayed relations make the two columns equal. Face additivity is
built into the assignment, and the face variables agree across pairings.
We therefore obtain a cellular cocycle on $X$. Conversely, recover the
local face values from a cocycle by
\[
 (x_0,x_1,x_2,x_3)=(h_{12},h_{02},h_{01},h_{01}),\qquad
 (y_0,y_1,y_2,y_3)=(-h_{23},-h_{23},-h_{13},-h_{12}).
\]
The triangle cocycle relations give exactly the five local relations
above. These constructions use only addition and sign changes, so they
work over every field. Hence
\[
 \ker(F_0\otimes k)=Z^1(X;k)=0.
\]
The last equality follows from $H_1(X;\Z)=0$, the universal coefficient
theorem, and the vanishing of the coboundary from the single vertex.

It follows that $F_0$ has column rank $4N$ over $\Q$ and modulo every
prime. If any of its $4N$ nonzero Smith factors were divisible by a
prime $p$, its rank modulo $p$ would drop. All Smith factors are
therefore $1$, which proves both assertions.
\end{proof}

Suppose $t,s\in\Z^N$ satisfy $At-Bs=0$. Rational duality gives rational
$x,y$ such that
\[
 F_0\binom{x}{y}
 =\begin{pmatrix}Dt\\Dt\\s-\Delta t\end{pmatrix}.
\]
The right side is integral, since the diagonal entries of $\Delta$ are
$0$ or $1$. Lemma~\ref{lem:primitive-lifts} supplies integral $x,y$.
Conversely, integral face data give an integral solution of $At-Bs=0$.
Thus
\begin{equation}
 \Gamma=\Lambda,\qquad
 \Lambda=\{t\in\Z^N:At\in B\Z^N\}.
 \label{eq:latticeeq}
\end{equation}

\subsection{The longitude lattice index}

\begin{lemma}\label{lem:longitude-index}
For the unhalved preferred-longitude row,
\begin{equation}
 \Z^N/(A\Z^N+B\Z^N)\cong\Z/2\Z.
 \label{eq:nzcoker}
\end{equation}
\end{lemma}
\begin{proof}
We spell out the lattice normalization, since rescaling the peripheral
row changes this conclusion. Let $\mathscr T\subset\Z^{3N}$ be
generated by $(1,1,1)$ in each tetrahedron, and put
$J=\Z^{3N}/\mathscr T$. The quotient map in each tetrahedron is
\[
 (g,g',g'')\longmapsto(g-g',g''-g').
\]
Thus $J\cong\Z^{2N}$ is the row-coordinate lattice for $(A\mid B)$.
Identify each shape coordinate with the quadrilateral type facing its
opposite-edge pair, using the table in Section~1.1. Thus the ordered
pairs $01\mid23,02\mid13,03\mid12$ correspond to $(z,z'',z')$ in a
positive tetrahedron and $(z,z',z'')$ in a negative tetrahedron. This
identifies the lattice by edge pairs, independently of the prime names
used for quadrilateral coordinates in the reference.

Let $Q\subset\Z^{3N}$ be the lattice of integral solutions of the
quadrilateral matching equations, and let $\mathscr E$ be generated by
the edge solutions. The latter are precisely the full gluing-row vectors
before the quotient by $\mathscr T$; see \cite[Sections 3.1 and
6.2]{GHHHR}. Since $\mathscr T\subset Q$, the matching map descends
to $J$, with kernel $\overline Q=Q/\mathscr T$. In particular
$\overline Q$ is primitive in $J$.

Theorem~7.1 of \cite{GHHHR} identifies
\[
 Q/(\mathscr E+\mathscr T)
 \cong\{(a,b)\in H_2(M,\partial M;\Z/2)
                   \times H_1(\partial M;\Z):
                   \partial_*a=b\bmod2\}.
\]
Choose a meridian--longitude basis $(\mu,\lambda)$, with $\mu$ mapping
to a generator of $H_1(M;\Z)$ and $\lambda$ mapping to zero. For a
knot exterior, the mod-two boundary map is injective with image generated
by $\lambda\bmod2$. Projection onto $b$ therefore identifies this
quotient with
\[
 L_{\partial}=2\Z\mu\oplus\Z\lambda.
\]
In particular, the edge lattice
$\overline{\mathscr E}=(\mathscr E+\mathscr T)/\mathscr T$
is primitive in $\overline Q$.

The signed-corner construction adds one signed quadrilateral for every
normal arc of an oriented peripheral curve of class $b$. With the
matching turning and boundary orientations, its boundary is $2b$,
by \cite[Section 7.1]{GHHHR}; this is also the normalization
$\gamma\delta=2\operatorname{id}$ in
\cite[Lemma 4.3(i)]{Neumann}. Our unhalved
longitude row is this construction for $b=\lambda$, so its image in
$L_{\partial}$ is $2\lambda$, up to a simultaneous reversal of the
peripheral orientation convention. The permitted additions of edge
and tetrahedral rows do not change this class in
$Q/(\mathscr E+\mathscr T)$. Let $L\subset J$ be generated by
$\overline{\mathscr E}$ and this longitude row. Modulo
$\overline{\mathscr E}$, it is $2\Z\lambda$, whose saturation in
$L_{\partial}$ is $\Z\lambda$. Thus
\[
 [\operatorname{sat}_{\overline Q}L:L]=2.
\]
Because $\overline Q$ is primitive in $J$, saturation in $J$ gives the
same group. Finally, the sum of all edge rows is zero in $J$: this sum
has the local coordinate triple $(2,2,2)$ in every tetrahedron.
Deleting any one edge row therefore leaves the generated lattice
unchanged. The rows of $(A\mid B)$ generate exactly $L$.

This row lattice has rank $N$ by \eqref{eq:nzfacts}. Its index in its
saturation is the product of the nonzero Smith factors of $(A\mid B)$,
so that product is $2$. The column cokernel of this full-row-rank matrix
is consequently the group of order $2$, proving \eqref{eq:nzcoker}.
\end{proof}

\subsection{The exact sequence and kernel decomposition}

Define
\[
 g:\Z^N\longrightarrow\coker B,\qquad g(t)=[At]_B.
\]
Its kernel is $\Lambda=\Gamma$, and Lemma~\ref{lem:longitude-index}
identifies its cokernel with $\Z/2\Z$. The map $\alpha$ above is
surjective with the same kernel. Therefore it induces an isomorphism
between $\coker C\oplus\coker\Ct$ and $\im g$. Explicitly, the
resulting injection is
\[
 \iota(([Dt]_C,[Dt]_{\Ct}))=[At]_B.
\]
Every element of its domain has this form because $\alpha$ is onto.
If $t$ and $t'$ represent the same element, then $t-t'\in\Gamma=
\Lambda$, so $A(t-t')\in B\Z^N$. This proves independence of the
representative. We obtain
\[
 0\longrightarrow\coker C\oplus\coker\Ct
 \xrightarrow{\iota}\coker B
 \longrightarrow\Z/2\Z\longrightarrow0,
\]
as asserted in \eqref{eq:cokersequence}. No finiteness or nonsingularity
assumption has been used.

For $s\in\ker_{\Z}B$, the integral point $(0,s)$ has an integral face
lift by Lemma~\ref{lem:primitive-lifts}. Hence
\[
 s=Rx+\Rt y,\qquad x\in\ker_{\Z}C,\quad
 y\in\ker_{\Z}\Ct.
\]
The injectivity of $F_0$ makes this representation unique. This proves
\eqref{eq:integerkernel} and completes Theorem~\ref{thm:integral}.

If $C$ is nonsingular, nullity doubling makes $B$ nonsingular as well.
Taking finite group orders in the exact sequence and applying
Theorem~\ref{thm:integral-reversal} gives
$|\det B|=2|\det C|^2$. If $C$ is singular, both determinants are zero.
This proves Theorem~\ref{thm:det}.

These arguments also apply to an oriented integral homology solid torus.
Its ideal compactification has $H_1(X;\Z)=0$, and its primitive
homological longitude generates the kernel of
$H_1(\partial M;\Z)\to H_1(M;\Z)$. The same mod-two boundary
calculation and the same doubled-longitude construction apply. This
extension does not cover rational homology solid tori with torsion in
$H_1(M;\Z)$.

\begin{corollary}[Odd-primary doubling and characteristic two]
\label{cor:primes}
At each odd prime $p$, the $p$-primary torsion of $\coker B$ is the
direct sum of two copies of that of $\coker C$. In particular,
\[
 \nullity_{\mathbb F_p}B=2\nullity_{\mathbb F_p}C
 \qquad(p\text{ odd}).
\]
In characteristic two,
\[
 2\nullity_{\mathbb F_2}C
 \ \leq\ \nullity_{\mathbb F_2}B
 \ \leq\ 2\nullity_{\mathbb F_2}C+1.
\]
\end{corollary}
\begin{proof}
Localization at an odd prime kills the final term of
\eqref{eq:cokersequence}; reversal identifies the first two summands.
For $p=2$, tensoring that sequence with $\mathbb F_2$ gives a connecting
map from $\operatorname{Tor}_1^{\Z}(\Z/2,\mathbb F_2)$. Its image
has dimension zero or one, which gives the two bounds.
\end{proof}

In the broader integral-homology-solid-torus setting, the ancillary
specimen \texttt{m121} has $\coker C\cong\Z/4\Z$ and
$\coker B\cong\Z/4\Z\oplus\Z/8\Z$. It shows that
\eqref{eq:cokersequence} need not split. This specimen is not asserted
to be a knot exterior in $S^3$. In the singular case, the extension can
also involve the free part, so the exact sequence does not by itself
produce a separate $2$-torsion summand.

\section{An embedded relative complex and cyclic face torsion}
\label{sec:embedded}

We prove Theorem~\ref{thm:cyclic} by realizing $C^{\mathsf T}$ as a
relative cellular boundary map. Throughout this section $M$ may be any
oriented integral homology solid torus with a one-cusped ordered ideal
triangulation. No angle structure is required.

\subsection{The local disks}

Truncate the ideal vertices. In each quotient face $f$, choose an arc
$u_f$ parallel to its lower edge, oriented from its lowest to its middle
cusp corner. Its two endpoints lie in the interiors of the corresponding
truncation edges. Make the choices on quotient faces and lift them to
local faces; this ensures agreement across every face pairing.

In one truncated tetrahedron, form two closed curves on its boundary.
The first follows $u_2,u_0,-u_1$, joined in the cusp caps at vertices
$1,2,0$. The second follows $u_2,-u_3$, joined in caps $1,0$. More
explicitly, the required joining arcs are
\[
\begin{array}{c|cc}
 \text{cap}&\text{first curve}&\text{second curve}\\\hline
 0&\text{sides }1\text{ and }2&\text{sides }2\text{ and }3\\
 1&\text{sides }2\text{ and }0&\text{sides }2\text{ and }3\\
 2&\text{sides }0\text{ and }1&\text{none}\\
 3&\text{none}&\text{none}
\end{array}
\]
Here a side is labeled by the face containing it. In caps $0$ and $1$,
the two arcs can be chosen with disjoint interiors and with their common
endpoint on side $2$. All other joining arcs lie in different cap
interiors. The two closed curves consequently share exactly $u_2$ and
form an embedded theta graph on the boundary sphere.

The two cycles containing $u_2$ bound the two complementary disk regions
incident to that edge. Choose a collar of the boundary sphere inside the
tetrahedron. On each of these disk regions choose a piecewise linear
height function that is zero on the boundary and positive in the
interior. The graphs of the two functions in the collar are embedded
disks with disjoint interiors, meeting exactly along $u_2$. This
construction fixes the boundary curves and moves all disk interiors
into the tetrahedron interior. Orient the disks independently so that,
modulo their cusp arcs, their boundaries are
\[
 u_2+u_0-u_1,\qquad u_2-u_3.
\]

Glue these patches across the paired faces and denote their union by
$\mathcal B_C\subset M$. The result is an embedded finite PL complex.
Indeed, disk interiors lie in tetrahedron interiors, which are not
identified by face gluings. At a quotient face, pages from the two
tetrahedral sides lie in its two distinct half-neighborhoods and meet
only along $u_f$. At an endpoint of $u_f$, exactly two cusp-triangle
sides are paired. Since the endpoint is in the interior of such a side,
there are no additional identifications through cusp vertices.
These descriptions also apply when two distinct faces of the same
tetrahedron are paired. Individual cell closures may acquire boundary
identifications; the resulting cellular structure need not be regular.
Set
\[
 \mathcal G_C=\mathcal B_C\cap\partial M.
\]
This is the embedded graph formed by the cusp arcs.

\subsection{The relative cellular complex}

The relative cells of $(\mathcal B_C,\mathcal G_C)$ are the $2N$ face
arcs $u_f$ in dimension one and the two disks per tetrahedron in
dimension two. There are no relative zero-cells and no cells of higher
dimension. With the first disk of each tetrahedron listed before the
second family, the boundary map is $C^{\mathsf T}$. Therefore
\begin{equation}
 H_1(\mathcal B_C,\mathcal G_C;\Z)=\coker C^{\mathsf T},\qquad
 H_2(\mathcal B_C,\mathcal G_C;\Z)=\ker C^{\mathsf T}.
 \label{eq:relativeface}
\end{equation}
The complex $\mathcal B_C$ is connected: each tetrahedral patch is
connected and includes all four local face arcs, and adjacent
tetrahedral patches share the arc of their paired face.

\subsection{Filling and divisibility}

Choose a solid-torus filling $Y=M\cup V$ that is an integral homology
sphere. For a knot exterior in $S^3$, take its original meridional
filling. Such a filling also exists for every integral homology solid
torus: Poincar\'e--Lefschetz duality gives
$H_1(M,\partial M;\Z)=H^2(M;\Z)=0$, so the boundary map onto
$H_1(M;\Z)=\Z$ is surjective. Any slope mapping to a generator is
primitive, and filling along it gives $H_1(Y;\Z)=0$ by the
Mayer--Vietoris sequence. Closed oriented three-manifold duality then
makes $Y$ an integral homology sphere.

Put
\[
 \mathcal Z=\mathcal B_C\cup V\subset Y.
\]
Since $\mathcal B_C\cap V=\mathcal G_C$, the relative cellular
complex of $(\mathcal Z,V)$ is that of
$(\mathcal B_C,\mathcal G_C)$.

We need the following elementary consequence of duality: the first
integral homology of an embedded finite PL complex in an integral
homology three-sphere is free abelian. In the present case, choose a
closed regular neighborhood $U$ of $\mathcal Z$ in $Y$ and put
$W=\overline{Y\setminus U}$. The regular neighborhood retracts onto
$\mathcal Z$. The pair sequence for $(Y,U)$, excision across collars,
and Poincar\'e--Lefschetz duality give
\[
 H_1(\mathcal Z;\Z)\cong H_1(U;\Z)
 \cong H_2(Y,U;\Z)
 \cong H_2(W,\partial W;\Z)
 \cong H^1(W;\Z).
\]
The pair-sequence isomorphism uses $H_2(Y)=H_1(Y)=0$. Finally,
$H^1(W;\Z)$ is free abelian by the universal coefficient theorem; this
does not require $W$ to be connected.

Both $V$ and $\mathcal Z$ are connected. Their pair sequence and
\eqref{eq:relativeface} consequently give
\begin{equation}
 \coker C^{\mathsf T}
 \cong H_1(\mathcal Z;\Z)/\im H_1(V;\Z).
 \label{eq:corequotient}
\end{equation}
The image of $H_1(V;\Z)$ is generated by the class of the filling core.
Write $H_1(\mathcal Z;\Z)\cong\Z^r$ and let $v$ be this core
class. If $v=0$, the quotient is free. Otherwise, write $v=dv_0$ with
$v_0$ primitive and $d>0$. Extending $v_0$ to an integral basis gives
\[
 \Z^r/\Z v\cong\Z^{r-1}\oplus\Z/d\Z.
\]
Thus the torsion subgroup of $\coker C^{\mathsf T}$ is cyclic.
Transposition preserves Smith normal form, so the same is true of
$\coker C$.

If $C$ is nonsingular, \eqref{eq:relativeface} also gives
$H_2(\mathcal Z,V;\Z)=0$. The pair sequence then makes
$H_1(V;\Z)\to H_1(\mathcal Z;\Z)$ injective. Its cokernel is finite,
so $H_1(\mathcal Z;\Z)$ has rank one and the core class is nonzero.
Its divisibility $d$ satisfies
\[
 \coker C^{\mathsf T}\cong\Z/d\Z,\qquad d=|\det C|.
\]
This proves Theorem~\ref{thm:cyclic}.

For nonsingular $C$, unimodularity is therefore equivalent to primitivity
of this filling-core class. The argument does not prove that primitivity.
In particular, rank-one homology does not imply that a regular
neighborhood of $\mathcal Z$ is a solid torus. A further restriction
from the geometry or ordering would be needed to deduce the original
unit-determinant assertion.

\section{Computational checks and remaining questions}\label{sec:checks}
\subsection{Reproducibility and scope of the checks}
The ancillary files contain executable exact-arithmetic checks, fixed input data, recorded output and pinned package versions. The checks reconstruct face matrices from ordered face tables, independently of the FAMED implementation used in \cite{BGW}. SnapPy supplies the gluing exponents and the fundamental-group presentation. Abelianization identifies the primitive rational longitude; for the knot exteriors this is the preferred longitude. The checks were run with Python~3.13, SnapPy~3.3.2, Regina~7.4.1 and SymPy~1.14.0.

\begin{center}
\begin{tabular}{lp{0.59\textwidth}}
\hline
Check & Extent\\
\hline
Local cochains & 48 polynomial identities, including both tetrahedron signs\\
Affine coordinates & The figure-eight example, including its logarithmic constants\\
Fixed specimens & 10 ordered triangulations, including face nullities $0,1,2$\\
Small face pairings & All $10{,}503$ labelled ordered tables with at most three tetrahedra\\
\hline
\end{tabular}
\end{center}

The fixed-specimen checks include the entire projected face relation, the constrained symmetric forms, integer-kernel decompositions, Smith factors, the determinant formula and reductions modulo $2,3,5,7,11$. Some specimens are integral homology solid tori with nonunit face determinant. They are not asserted to be knot exteriors in $S^3$ and therefore do not contradict the determinant conjecture in \cite{BGW}.

Of the $10{,}503$ labelled tables, $712$ give connected valid orientable one-torus ideal triangulations with first Betti number one. Among these, $43$ have no edge of degree one or two, and $31$ admit a strict angle structure. The positive-rank check passes on every table in these latter classes. These are counts of labelled tables, not counts of distinct manifolds. Regina's strict-angle test checks an angle structure; it is not a certificate of a complete geometric shape solution.

The ancillary \texttt{README.md} gives four commands that reproduce these checks. The two regenerated main result files agree byte-for-byte with the recorded outputs. Ranks, kernels, determinants, and matrix identities are decided by exact arithmetic. Numerical shape approximations play no role in these decisions. The finite computations provide checks on conventions and possible counterexamples; the universal claims depend on the proofs above.

\subsection{Questions left open}
Theorem~\ref{thm:det} reduces the remaining determinant assertion in Conjecture~4.1 of \cite{BGW} to $|\det C|=1$ in its nonsingular geometric knot setting. Section~\ref{sec:embedded} interprets this condition as primitivity of a filling-core class, but does not establish that primitivity. Proving it requires information beyond the cyclicity of the face torsion.

Wong's longitude--meridian conditions in Definition~1.5 of \cite{Wong} remain additional hypotheses for the Jones-function results. The preferred-longitude relation proved here does not supply these additional conditions. Finally, this work does not construct a suitable geometric triangulation for every hyperbolic knot. The full Andersen--Kashaev volume conjecture consequently remains outside the conclusions of this paper.

\appendix
\section{AI tools and verification status}
OpenAI Codex was used extensively to propose and develop the mathematical arguments, draft and revise the manuscript, generate computational checks, search the literature, and conduct internal reviews. The author directed the investigation and preparation of the manuscript. The use of Codex extended to the mathematical development and was not limited to language editing.

Partial informal specialist feedback has been received on an earlier draft. This does not constitute a complete independent verification of the arguments, and no formal proof-assistant verification has been completed. Internal reviews and finite exact-arithmetic checks do not establish correctness or novelty. This preprint is circulated for critical scrutiny; responsibility for its contents rests with the named author.
\end{document}